\documentclass{birkmult}

\usepackage{color}
\usepackage{verbatim}

 \newtheorem{thm}{Theorem}[section]
 
 \newtheorem{lem}[thm]{Lemma}
 
 \theoremstyle{definition}
 
 \theoremstyle{remark}
 \newtheorem{rem}[thm]{Remark}
 
 \numberwithin{equation}{section}
 \newcommand{\F}{\mathcal{F}}

\begin{document}
\setlength{\baselineskip}{16pt}
%
%
%
%
%
%
%
%
%
\title[Fundamental solutions for higher order Schr\"odinger equations]
{Global estimates of fundamental solutions for higher-order
Schr\"odinger equations}
\author{JinMyong Kim}

\address[Current Address]{%
: Institut f\"{u}r Analysis und Scientific Computing, \\
Technische Universit\"{a}t Wien\\
Wiedner Hauptstr. 8, A-1040 Wien, Austria;\\[2mm]
(Permanent Address) : Department of Mathematics,\\
Kim Il Sung University\\
Pyongyang,  DPR Korea;} \email{jinjm39@yahoo.com.cn}

\author{Anton Arnold}

\address{Institut f\"{u}r Analysis und Scientific Computing, \\
Technische Universit\"{a}t Wien\\
Wiedner Hauptstr. 8, A-1040 Wien, Austria;} \email{anton.arnold@tuwien.ac.at}

\author{Xiaohua Yao }
\address{ Department of Mathematics\\
Central China Normal University \\
Wuhan 430079, P. R. China;} \email{yaoxiaohua@mail.ccnu.edu.cn} 

\thanks{This work was supported by the Postdoctoral Science Foundation of Huazhong University of Science and Technology in China and the Eurasia-Pacific Uninet scholarship for post-docs in Austria.
The second author was supported by the FWF (project I 395-N16). {The third author
was supported by NSFC (No. 10801057), the Key Project of Chinese
Ministry of Education (No. 109117), NCET-10-0431, and  CCNU Project (No. CCNU09A02015)}}





\date{September 4, 2011}
\dedicatory{}

\begin{abstract}
  In this paper we first establish global pointwise time-space
estimates of the fundamental solution for Schr\"odinger equations,
 where the symbol of the spatial operator is a real non-degenerate elliptic
polynomial. Then we use such estimates to establish
related $L^p-L^q$ estimates on the Schr\"odinger solution. These estimates extend known results from the literature and are sharp. This result was lately already generalized
to a degenerate case (cf. \cite{DY}).\end{abstract}

\maketitle


\vspace{-5pt}

\section{Introduction} \label{S1}

In this paper we are interested in $L^p$-$L^q$ estimates of solutions
for the following Schr\"odinger equation:
\begin{equation} \label{eq:11}
\frac{\partial u}{\partial t}=iP(D)u,  u(0,\cdot)=u_0\in L^p({\bf
R}^n),
\end{equation}
where $D=-i(\partial/\partial x_1,\cdots,\partial/\partial x_n)$,
$P:{\bf R}^n\longrightarrow {\bf R}$ is a non-degenerate real elliptic
polynomial of the even order $m$. In the sequel, we may assume
without loss of generality that $P_m(\xi)>0$ for $\xi\ne0$ where
$P_m(\xi)$ is the principal part of $P(\xi)$. The non-degeneracy
condition on the polynomial $P$ reads as follows.

(a) For any $\xi\in {\bf R}^n\backslash\{0\}$ the Hessian
$$
\left(\frac{\partial^{2}}{\partial\xi_{i}\partial\xi_{j}}P_m(\xi)\right)
$$
is non-degenerate.

For an elliptic polynomial $P$, condition (a) is equivalent to the
following condition (see \cite{BE}):

(b) For any $z\in {\bf S}^{n-1}$(the unit sphere of ${\bf R}^n$), the function
on $ {\bf S}^{n-1}$ $\psi(\omega):=$ $<z,\omega>(P_m(\omega))^{-1/m}$, where
$\omega\in  {\bf S}^{n-1}$, is non-degenerate at its critical points.
 This means, if $ d_{\omega}\psi$, the differential of $\psi$ at a point $\omega\in {\bf S}^{n-1}$
vanishes, then $ d_{\omega}^{2}\psi$, the second order differential of $\psi$ at this point
is non-degenerate.

For every initial data $u_{0}\in  S({\bf R}^n)$ (the Schwarz space),
the solution of the Cauchy problem \eqref{eq:11} is given by
$$
u(t,\cdot)=e^{itP(D)}u_{0}= {\F}^{-1}(e^{itP})\ast u_{0},
$$
where ${\F}$ denotes the Fourier transform, $ {\F}^{-1}$ its
inverse, and $ {\F}^{-1}(e^{itP})$ is understood in the
distributional sense. From the  ellipticity assumption on $P$, it
is easy to find that $I(t,x):={\F}^{-1}(e^{itP})(x)$ is an infinitely
differentiable function in the $x$ variable for every fixed $t\neq 0$ (see
\cite{C2}).

When the symbol $P$ is homogeneous, Miyachi \cite{Miy} and Zheng et al. \cite{ZYF}
considered the pointwise estimates of the oscillatory integral $I$ and the $L^p-L^q$
estimates of the operator $e^{itP(D)}(t\neq 0)$.
Dropping the homogeneity of $P$, Balabane et al. \cite{BE} and Cui \cite{C1,C2}
studied the same estimates under the above non-degeneracy condition. We remark that the results of Balabane et al. are not sharp, while those of Cui are sharp estimates,  but under the assumption of local $t$, i.e. $0<|t|<T$. Here,
sharpness means that the decay rate in the spatial variable is identical with that in the homogeneous case, namely,
the decay rate is $-\frac{n(m-2)}{2(m-1)}$ (see
\cite{ZYF}).

The purpose of this paper is to prove global
pointwise time-space estimates and $L^p-L^q$ estimates of the
fundamental solution of \eqref{eq:11} for
all $|t|>0$.
Our proof depends heavily on a decay estimate for the oscillatory
integral ${\F}^{-1}(e^{itP})$. Compared with previous papers (see
\cite{BE,C2,KPV,YZ,ZYF}), we estimate the oscillatory integral with two parameters,
i.e. both the time variable and the spatial variable simultaneously. So we
obtain the sharp decay in the spatial variable, even for $|t|$ large.
Recently, our result was already generalized in \cite{DY}. 
But since the method applied there is different, this paper provides an alternative approach.

This paper is organized as follows. In Section \ref{S2}, we make some
pretreatment of the oscillatory integral ${\F}^{-1}(e^{itP})$,
review the method of Balabane et al. \cite{BE} and Cui \cite{C2}, and present
some necessary lemmata. In Section \ref{S3}, we prove global
pointwise time-space estimates of the fundamental solution of
\eqref{eq:11} which is our main result. Finally, in \S\ref{S4}
we use them to obtain the related  $L^p-L^q$ estimates for the Schr\"odinger solution.

\section{Preliminaries} \label{S2}

Throughout this paper, we assume that $P:{\bf R}^n\to{\bf R}$ is
always a non-degenerate elliptic inhomogeneous polynomial of order
$m$ where $n\ge2$ and $m$ is even. It is clear that $P$ is
non-degenerate if and only if det$(\partial_i\partial_j
P(\xi))_{n\times n}$ is an elliptic polynomial of order $n(m-2)$,
which is also equivalent to $(H2)$ in \cite{BE}, i.e. our condition (b).

We denote by ${\bf S}^{n-1}$ the unit sphere in ${\bf R}^n$, and by
$(\rho,\omega)\in[0,\infty) \times{\bf S}^{n-1}$ the polar
coordinates in ${\bf R}^n$. By the conditions on $P$,  we know that
$P_m(\xi)>0$ for $\xi\ne0$, which implies that there exists a large constant
$a>0$ with: For each fixed $s\ge a$ and $\omega\in{\bf
S}^{n-1}$, the equation $P(\rho\omega)=s$  has an unique positive solution $\rho=\rho(s,\omega)\in
C^\infty([a,\infty)\times{\bf S}^{n-1})$. By Lemma 2 in \cite{BE} we have
\begin{equation} \label{eq:21}
\rho(s,\omega)=s^{1\over m}(P_{m}(\omega))^{-{1\over m}}
+\sigma(s,\omega),
\end{equation}
where $\sigma$ lies in the symbol class $ S^0_{1,0}([a,\infty)\times{\bf S}^{n-1})$ (cf.  \cite{DY}), i.e.
$\sigma\in C^\infty ([a,\infty)\times{\bf S}^{n-1})$. Moreover for every
$k\in{\bf N}_0:=\{0,1,2,\cdots\}$ and every differential operator
$L_\omega$ on ${\bf S}^{n-1}$ there exists a constant $C_{kL}$ such
that
\begin{equation} \label{eq:22}
|\partial^k_s L_\omega\sigma(s,\omega)|\le C_{kL}(1+s)^{-k}
\quad{\rm for}\ s\ge a\ {\rm and}\ \omega\in{\bf S}^{n-1}.
\end{equation}
\vskip0.3cm We now recall two lemmata (see \cite{BE,C2}) on the
estimates of the following phase function
$$
\phi(s,\omega):=s^{-{1\over m}}\rho(s,\omega)\langle u,\omega
\rangle\quad{\rm for}\ s\ge a\ {\rm and}\ \omega\in{\bf S}^{n-1},
$$
with any fixed $u\in{\bf S}^{n-1}$. Since for every fixed $u_0\in{\bf
S}^{n-1}$ there exists a sufficiently small neighborhood
$U_{u_0}\subset{\bf S}^{n-1}$ of $u_0$ such that the following
lemmata always hold uniformly in $u\in U_{u_0}$(i.e. the constants in Lemma \ref{lem:21} and Lemma \ref{lem:22} are independent of $u$) we do not put the
variable $u$ in the function $\phi$. Clearly, $\phi\in
S^0_{1,0}([a,\infty)\times{\bf S}^{n-1})$.

\begin{lem} \label{lem:21}
There exists a constant $a_0\ge a$ and an open cover
$\{\Omega_0,\Omega_{\mbox{\tiny +}},\Omega_{\mbox{-}}\}$ of ${\bf
S}^{n-1}$ with $\Omega_{\mbox{\tiny
+}}\cap\Omega_{\mbox{-}}=\emptyset$ such that for $s\ge a_0$,

(a) The function $\Omega_0\ni\omega\mapsto\phi(s,\omega)$ has no
critical points, and
\begin{equation} \label{eq:23}
\|d_{\omega}\phi(s,\omega)\|\ge c>0\quad{\rm for}\
\omega\in\Omega_0,
\end{equation}
where the constant $c$ is independent of $s$.

(b) The function $\Omega_{\mbox{\tiny
$\pm$}}\ni\omega\mapsto\phi(s,\omega)$ has a unique critical point\\
$\omega_{\mbox{\tiny $\pm$}}\in C^\infty([a_0,\infty);
\Omega'_{\mbox{\tiny $\pm$}})$ for some open subset
$\Omega'_{\mbox{\tiny $\pm$}}$ with $\overline{\Omega'}_{\mbox{\tiny
$\pm$}}\subset\Omega_{\mbox{\tiny $\pm$}}$, respectively. Furthermore
\begin{equation} \label{eq:24}
\|(d^2_{\omega}\phi(s,\omega))^{-1}\|\le c_0 \quad{\rm for}\
\omega\in\Omega_{\mbox{\tiny $\pm$}},
\end{equation}
where the constant $c_0$ is independent of $s$. Moreover,
$\lim_{s\to\infty}\omega_{\mbox{\tiny $\pm$}}(s)$ exist and
$$
|\omega_{\mbox{\tiny $\pm$}}^{(k)}(s)|\le c_k(1+s)^{-k-{1\over m}}
\quad{\rm for}\ k\in{\bf N}.
$$
\end{lem}
\vskip0.3cm
\begin{lem} \label{lem:22}
 Let $\phi_{\mbox{\tiny $\pm$}}(t,r,s)=st+rs^{1\over
m}\phi (s,\omega_{\mbox{\tiny $\pm$}}(s))$ for $t$, $r>0$ and $s\ge
a$. Then there exist constants $a_1\ge a_0$ and $c_2>c_1>0$ such
that for $s\ge a_1$, $t>0$, and $r>0$,
\begin{equation} \label{eq:25}
 c_1\le\pm\phi(s,\omega_{\mbox{\tiny
$\pm$}}(s))\le c_2,
\end{equation}
\begin{equation} \label{eq:26}
\partial_s\phi_{\mbox{\tiny +}}(t,r,s)\ge t+c_1rs^{{1\over
m}-1},
\end{equation}
\begin{equation} \label{eq:27}
t-c_2rs^{{1\over m}-1}\le\partial_s\phi_{\mbox{-}}(t,r,s)\le
t-c_1rs^{{1\over m}-1},
\end{equation}
\begin{equation} \label{eq:28}
c_1rs^{{1\over m}-2}\le|\partial^2_s\phi_{\mbox{-}}(t,r,s)|\le
c_2rs^{{1\over m}-2},
\end{equation}
 and
\begin{equation} \label{eq:29}
|\partial^k_s\phi_{\mbox{\tiny $\pm$}}(t,r,s)|\le c_2rs^{{1\over
m}-k}\quad{\rm for}\ k=2,3,\cdots.
\end{equation}
\end{lem}
\vskip0.4cm Next, we consider the following oscillatory integral
$$
\Phi(\lambda,s)=\int_{{\bf S}^{n-1}}e^{i\lambda\phi(s,\omega)}
b(s,\omega)d\omega,
$$
where $b(s,\omega):=s^{1-{n\over m}}\rho^{n-1}\partial_s\rho\in
S^0_{1,0} ([a,\infty)\times{\bf S}^{n-1})$. Let
$\varphi_{\mbox{\tiny +}}$, $\varphi_{\mbox{-}}$, $\varphi_0$ be a
partition of unity of ${\bf S}^{n-1}$, subordinate to the open
cover given in Lemma \ref{lem:21}. Then
$$
\Phi(\lambda,s)=\Phi_{\mbox{\tiny +}}(\lambda,s)+\Phi_{\mbox{-}}
(\lambda,s)+\Psi_0(\lambda,s),
$$
where
$$
\Phi_{\mbox{\tiny $\pm$}}(\lambda,s)=\int_{{\bf
S}^{n-1}}e^{i\lambda\phi(s,\omega)} b(s,\omega)\varphi_{\mbox{\tiny
$\pm$}}(\omega)d\omega
$$
and
$$
\Psi_0(\lambda,s)=\int_{{\bf S}^{n-1}}e^{i\lambda\phi(s,\omega)}
b(s,\omega)\varphi_0(\omega)d\omega.
$$
By using the stationary phase method for $\Psi_0$, and Lemma \ref{lem:21} and
Corollary 1.1.8 in \cite{So} for $\Phi_{\mbox{\tiny $\pm$}}$, one has the
following result.

\begin{lem} \label{lem:23}
For $\lambda>0$ and $s>a_1$ we have
\begin{equation} \label{eq:210}
\Phi(\lambda,s)=\lambda^{-\frac{n-1}{2}}e^{i\lambda
\phi(s,\omega_{\mbox{\tiny +}}(s))}\Psi_{\mbox{\tiny +}}(\lambda,s)
+\lambda^{-\frac{n-1}{2}}
e^{i\lambda\phi(s,\omega_{\mbox{-}}(s))}\Psi_{\mbox{-}}(\lambda,s)+\Psi_0(\lambda,s),
\end{equation}
where $\Psi_{\mbox{\tiny $\pm$}}$, $\Psi_0\in
C^\infty((0,\infty)\times[a_0,\infty))$ and
\begin{equation} \label{eq:211}
|\partial^k_\lambda\partial^j_s\Psi_{\mbox{\tiny
$\pm$}}(\lambda,s)|\le c_{k,j}(1+\lambda)^{-k}s^{-j}\quad {\rm for}\
k,j\in{\bf N}_0,
\end{equation}
\begin{equation} \label{eq:212}
|\partial^k_\lambda\partial^j_s\Psi_0(\lambda,s)|\le c_{k,j,l}
(1+\lambda)^{-l}s^{-j}\quad{\rm for}\ k,j,l\in{\bf N}_0.
\end{equation}
\end{lem}

\section{Estimates on the oscillatory integral}
\label{S3}

 In this section we establish the global pointwise time-space estimates
of the fundamental solution for the Schr\"odinger equation \eqref{eq:11}.

\begin{thm} \label{thm:31}
 If the inhomogeneous polynomial $P$ is elliptic and
non-degenerate, then the fundamental solution of \eqref{eq:11} satisfies
that there exists a constant $C>0$ such that
\begin{equation} \label{eq:31}
|I(t,x)|= | {\F}^{-1}(e^{itP})(x)|\le
\left\{
\begin{array}{ll}
    C|t|^{-\frac{n}{m}}(1+|t|^{-\frac{1}{m}}|x|)^{-\mu}  &\hbox{\rm for}\ \ 0<|t|\leq1, \\
    C|t|^{-\frac{1}{m}}(1+|t|^{-1}|x|)^{-\mu}  &\hbox{\rm for}\ \ |t|\ge 1,\\
\end{array}
\right.
\end{equation}
 where $\mu=\frac{n(m-2)}{2(m-1)}$.

\end{thm}

\begin{proof}
We first consider \\
\underline{Case (i):} $t\ge1$ and $r:=|x|\ge1$.\\
Let $\psi\in C^\infty({\bf R})$ such that $\psi(s)=
\left\{%
\begin{array}{ll}
    0 &\hbox{\rm for}\ s\le a_1 \\
    1 &\hbox{\rm for}\ s>2a_1\\
\end{array}
\right. , $ where $a_1$ is given in Lemma \ref{lem:22}. We write
\begin{eqnarray*}
I(t,x)&=& {\F}^{-1}(e^{itP})(x)=\int_{{\bf R}^n}e^{i(\langle
x,\xi\rangle+tP(\xi))}\psi(P(\xi))d\xi\\
&+&\int_{{\bf R}^n}
e^{i(\langle x,\xi\rangle+tP(\xi))}(1-\psi(P(\xi)))d\xi
=:I_1(t,x)+I_2(t,x).
\end{eqnarray*}
First we rewrite $I_2$ as the Fourier transform of a measure, supported on the graph $S:=\{z=P(\xi);\:\xi\in {\bf R}^n\}\subset {\bf R}^{n+1}$:
\begin{equation}\label{surface-meas}
  I_2(t,x)=\int_{{\bf R}^{n+1}}e^{i(\langle x,\xi\rangle + tz)} (1-\psi(P(\xi)))
  \delta(z-P(\xi))\,d\xi\,dz\,.
\end{equation}
Since the polynomial $P$ is of order $m$, the supporting manifold of the above integrand is of type less or equal $m$ (in the sense of \S~VIII.3.2, \cite{St}). Then, Theorem 2 of \S~VIII.3 in \cite{St} implies
\begin{equation} \label{eq:310a}
  |I_2(t,x)|\le C(1+|t|+|x|)^{-\frac1m}\qquad \forall t,\,x.
\end{equation}
This can be generalized: Since $f(t,\xi):=e^{itP}(1-\psi(P))\in C_c^\infty({\bf R}^n)$ for every $t>0$,
an integration by parts in $I_2$ yields
$$
  I_2(t,x)=i\int_{{\bf R}^n}
e^{i\langle x,\xi\rangle}\frac{x}{|x|^2}\cdot\nabla_\xi f(t,\xi) d\xi.
$$
Proceeding recursively, a simple estimate yields
$$
|I_2(t,x)|\le C_kt^kr^{-k}\quad{\rm for}\ k\in{\bf N}_0,
$$
and hence also $\forall\,k\ge0$.
But proceeding as in \eqref{surface-meas} yields the improvement
\begin{equation} \label{eq:32}
  |I_2(t,x)|\le C_k |t|^{-\frac1m}(1+|t|^{-1}|x|)^{-(k+\frac1m)}\ {\rm for}\ \
  |t|\ge 1, \ x\in{\bf R}^n, \ \forall \ k\ge0.
\end{equation}

\smallskip

To estimate $I_1$, we shall derive an $\varepsilon$--uniform estimate of its regularization
$$
J_\varepsilon(t,x):=\int_{{\bf R}^n}e^{-\varepsilon P(\xi)+i(\langle
x,\xi\rangle+tP(\xi))}\psi(P(\xi))d\xi\quad{\rm for}\ \varepsilon>0.
$$
By the polar coordinate transform and by the change of variables
$\rho=\rho(s,\omega)$ we have
\begin{eqnarray*}
J_\varepsilon(t,x)&=&\int^\infty_0\int_{{\bf
S}^{n-1}}e^{-\varepsilon P(\rho\omega)+i(\rho\langle
x,\omega\rangle+tP(\rho\omega))}
\psi(P(\rho\omega))\rho^{n-1}d\omega d\rho\\
&=&\int^\infty_0\int_{{\bf S}^{n-1}}e^{-\varepsilon s+its+
ir\rho\langle u,\omega\rangle}\psi(s)\rho^{n-1}\partial_s\rho d\omega ds\\
&=&\int^\infty_0 e^{-\varepsilon s+its}s^{\frac{n}
{m}-1}\psi(s)\Phi(rs^\frac{1}{m},s)ds,
\end{eqnarray*}
where $u=x/|x|$.

Due to the compactness of ${\bf S}^{n-1}$ we
may assume without loss of generality that $u\in U_{u_0}$ (see
section \ref{S2} for the definition of $U_{u_0}$). Thus by Lemma \ref{lem:23}
\begin{eqnarray*}
J_\varepsilon(t,x)&=&r^{-\frac{n-1}{2}}\int_0^\infty e^{-\varepsilon
s+i\phi_{\mbox{\tiny +}}(t,r,s)}
s^{\frac{n+1}{2m}-1}\psi(s)\Psi_{\mbox{\tiny +}}(rs^{\frac{1}{m}},s)ds\\
&&+r^{-\frac{n-1}{2}}\int_0^\infty e^{-\varepsilon
s+i\phi_{\mbox{-}}(t,r,s)}
s^{\frac{n+1}{2m}-1}\psi(s)\Psi_{\mbox{-}}(rs^{\frac{1}{m}},s)ds\\
&&+\int^\infty_0e^{-\varepsilon s+its}s^{\frac{n}{m}-1}\psi(s)
\Psi_0(rs^{\frac{1}{m}},s)ds\\
&=&R_\varepsilon^{\mbox{\tiny
+}}(t,x)+R_\varepsilon^{\mbox{-}}(t,x)+R_\varepsilon^0(t,x),
\end{eqnarray*}
where $\phi_{\mbox{\tiny $\pm$}}$ is the same as in Lemma \ref{lem:22}. In
the sequel, we denote by $C$ a generic positive constant independent
of $t$, $r$, $s$ and $\varepsilon$, and put $\mu:=\frac{n(m-2)}
{2(m-1)}$ and $\nu:=\frac{n}{2(m-1)}$.

We first estimate the integral $R_\varepsilon^0(t,x)$. Let
$v_0(s):=s^{\frac{n}{m}-1}\psi(s)\Psi_0(rs^{\frac{1}{m}},s)$. By the
Leibniz rule and \eqref{eq:212} one has
$$
|v_0^{(k)}(s)|\le C(rs^\frac{1}{m})^{-j}s^{\frac{n}{m}-1-k}\quad{\rm
for}\ j,k\in{\bf N}_0,
$$
where $r\ge1$ and $s\ge a_1$. Choosing $j\ge\mu$ and $k\ge\nu$,
it follows by integration by parts that
\begin{equation} \label{eq:33}
|R^0_{\varepsilon}(t,x)|\le Ct^{-k}\int^\infty_{a_1}(rs^{1\over
m})^{-j}s^{{n\over m}-1-k}ds\le Ct^{-k}r^{-j}\le Ct^{-\nu}r^{-\mu}.
\end{equation}

To estimate the integral $R_\varepsilon^{\mbox{\tiny +}}(t,x)$, for
given $t,r\ge1$ we set
$$
\left\{
\begin{array}{ll}
    u_{\mbox{\tiny +}}(s):=-\varepsilon s+i\phi_{\mbox{\tiny +}}(t,r,s) \\
    v_{\mbox{\tiny +}}(s):=s^{{n+1\over2m}-1}\psi(s)\Psi_{\mbox{\tiny  +}}(rs^{1\over
m},s)\\
\end{array}
\right. $$ for $s\ge a_1$. Since $u_{\mbox{\tiny +}}'(s)\ne0$ for
$s\ge a_1$, we can define $D_\ast f=(gf)'$ for $f\in C^1(0,\infty)$
where $g=-1/u_{\mbox{\tiny +}}'$. It is not hard to show
\begin{equation} \label{eq:34}
D^j_\ast v_{\mbox{\tiny +}}=\sum_\alpha c_\alpha
g^{(\alpha_1)}\cdots g^{(\alpha_j)}v_{\mbox{\tiny
+}}^{(\alpha_{j+1})} \quad{\rm for}\ j\in{\bf N}
\end{equation}
where the sum runs over all $\alpha=(\alpha_1,\cdots
\alpha_{j+1})\in{\bf N}^{j+1}_0$ such that $|\alpha|=j$ and
$0\le\alpha_1\le\cdots\le\alpha_j$. Since \eqref{eq:26} and \eqref{eq:29} imply,
respectively, that $|g(s)|\le Cr^{-1}s^{1-{1\over m}}$ and
$$
|u_{\mbox{\tiny +}}^{(k)}(s)|\le Crs^{{1\over m}-k}\quad{\rm for}\
k=2,3,\cdots,
$$
by induction on $k$ we find that
$$
|g^{(k)}(s)|\le Cr^{-1}s^{1-{1\over m}-k}\quad{\rm for}\ k\in{\bf
N}_0,
$$
which shall yield the spatial decay of $I_1$. To derive the time decay of $I_1$ we note that
\eqref{eq:26} also implies $|g(s)|\le t^{-1}$. Hence, it
follows that
$$
|g^{(k)}(s)|\le Ct^{-1}s^{-k}\quad{\rm for}\ k\in{\bf N}_0.
$$
The novel key step is now to interpolate these two inequalities, which will allow to derive estimates also for large time. We have
for any $\theta\in[0,1]$,
\begin{equation} \label{eq:35}
|g^{(k)}(s)|\le Ct^{\theta-1}r^{-\theta}s^{\theta(1-{1\over
m})-k}\quad{\rm for}\ k\in{\bf N}_0.
\end{equation}
 On the other
hand, by the Leibniz rule and \eqref{eq:211},
\begin{equation} \label{eq:36}
|v_{\mbox{\tiny +}}^{(k)}(s)|\le Cs^{{n+1\over 2m}-1-k}\quad{\rm
for}\ k\in{\bf N}_0.
\end{equation}
It thus follows from \eqref{eq:34} - \eqref{eq:36}
that
\begin{equation} \label{eq:37}
|D^j_\ast v_{\mbox{\tiny +}}(s)|\le Ct^{j(\theta-1)}
r^{-j\theta}s^{j\theta(1-{1\over m})+{n+1\over 2m}-1-j}\quad{\rm
for}\ j\in{\bf N}_0,
\end{equation}
where $D^0_\ast v_{\mbox{\tiny +}}=v_{\mbox{\tiny +}}$. Particularly
($\theta={\mu\over n}={m-2\over2(m-1)}$, $j=n$)
$$
|D^n_\ast v_{\mbox{\tiny +}}(s)|\le Ct^{\mu-n}
r^{-\mu}s^{-{nm+n-1\over 2m}-1}.
$$
Noting that $\mu-n<-\nu$, by integration by parts one gets that
$$
|R^{\mbox{\tiny +}}_\varepsilon(t,x)|=r^{-{n-1\over
2}}\Big{|}\int_0^\infty e^{u_{\mbox{\tiny +}}}(D^n_\ast
v_{\mbox{\tiny +}})ds\Big{|} \le Ct^{\mu-n}r^{-{n-1\over 2}-\mu}\le
Ct^{-\nu}r^{-\mu}.
$$

We now turn to the integral $R_\varepsilon^{\mbox{-}}(t,x)$. Here we put
$$
\left\{
\begin{array}{ll}
    u_{\mbox{-}}(s):=-\varepsilon s+i\phi_{\mbox{-}}(t,r,s) \\
    v_{\mbox{-}}(s):=s^{{n+1\over2m}-1}\psi(s)\Psi_{\mbox{-}}(rs^{1\over m},s)\\
\end{array}
\right. $$ for $s\ge a_1$, and write
\begin{eqnarray*}
R_\varepsilon^{\mbox{-}}(t,x)&=&r^{-{n-1\over
2}}\Big{\{}\int_0^{c'_1s_0}
+\int_{c'_1s_0}^{c'_2s_0}+\int_{c'_2s_0}^\infty\Big{\}}
e^{u_{\mbox{-}}(s)}v_{\mbox{-}}(s)ds\\
&=&R_{\varepsilon 1}^{\mbox{-}}(t,x)+R_{\varepsilon
2}^{\mbox{-}}(t,x)+R_{\varepsilon 3}^{\mbox{-}}(t,x),
\end{eqnarray*}
where $s_0=(r/t)^{m\over m-1}$, $c'_1=(c_1/2)^{m\over m-1}$, and
$c'_2=(2c_2)^{m\over m-1}$ ($c_1$ and $c_2$ are given in Lemma \ref{lem:22}).

By integration by parts one gets
$$
R^{\mbox{-}}_{\varepsilon 3}(t,x)=r^{-{n-1\over
2}}\Big{(}{e^{u_{\mbox{-}}(c'_2s_0)} \over u_{\mbox{-}}'(c'_2s_0)}
\sum^{n-1}_{j=0}(D^j_\ast
v_{\mbox{-}})(c'_2s_0)+\int_{c'_2s_0}^\infty
e^{u_{\mbox{-}}}(D^n_\ast v_{\mbox{-}})ds\Big{)}.
$$
Since \eqref{eq:27} implies that $|u_{\mbox{-}}'(s)|\ge c_2rs^{{1\over
m}-1}$ for $s\ge c'_2s_0$, we find that $v_{\mbox{-}}(s)$ still
satisfies \eqref{eq:37} (with $\theta=1$) for $s\ge c'_2s_0$.

If $c'_2s_0\le a_1$, then $(D^j_\ast
v_{\mbox{-}})(c'_2s_0)=0$ for $j=0,...,n-1$ (note that $\psi\equiv0$ on $[0,a_1]$). Integration by parts then yields
$$
|R^{\mbox{-}}_{\varepsilon 3}(t,x)|=\left|r^{-{n-1\over 2}}
\int_{c'_2s_0}^\infty e^{u_{\mbox{-}}}(D^n_\ast v_{\mbox{-}})ds\right|
\le Ct^{-\nu}r^{-\mu},
$$
exactly as done for $R^{\mbox{+}}_{\varepsilon}(t,x)$.
If $c'_2s_0> a_1$, then
\begin{eqnarray*}
|R^{\mbox{-}}_{\varepsilon 3}(t,x)|&\le& Cr^{-{n-1\over2}}\Big{(}
(rs_0^{{1\over m}-1})^{-1}\sum_{j=0}^{n-1}r^{-j}s_0^{-{2j-n-1\over
2m}-1}+\int_{c'_2s_0}^\infty r^{-n}s^{-{n-1\over 2m}-1}ds\Big{)} \\
&\le&Cr^{-{n-1\over2}}(r^{-1}s_0^{n-1\over2m}\sum_{j=0}^{n-1}
(rs_0^{1\over m})^{-j}+r^{-n} s_0^{-{n-1\over2m}}).
\end{eqnarray*}
Noting that $r\ge1$, $s_0\ge a_1/c'_2$, and $t\ge1$ it
follows that
$$
|R^{\mbox{-}}_{\varepsilon 3}(t,x)|\le
Cr^{-{n+1\over2}}s_0^{n+1\over2m}
=Ct^{-{n+1\over2(m-1)}}r^{-{(n+1)(m-2)\over2(m-1)}}\le
Ct^{-\nu}r^{-\mu}.
$$
Since $|u_{\mbox{-}}'(s)|\ge{1\over2}c_1rs^{{1\over m}-1}$ for
$a_1\le s\le c'_1s_0$, a slight modification of the above method
yields the same estimate for $R^{\mbox{-}}_{\varepsilon 1}(t,x)$.

To estimate $R_{\varepsilon 2}^{\mbox{-}}(t,x)$,
it suffices to estimate the integral
\begin{eqnarray*}
R_{0 2}
^{\mbox{-}}(t,x)&=&r^{-{n-1\over2}}\int_{c'_1s_0}^{c'_2s_0}
e^{i\phi_{\mbox{-}}(t,r,s)}v_{\mbox{-}}(s)ds\\
&=&r^{-{n-1\over2}}s_0\int_{c'_1}^{c'_2}e^{i\phi_{\mbox{-}}
(t,r,s_0\tau)}v_{\mbox{-}}(s_0\tau)d\tau.
\end{eqnarray*}
We note by \eqref{eq:28} that
$$
|\partial_\tau^2\phi_{\mbox{-}}(t,r,s_0\tau)|\ge
c_1rs_0^2(s_0\tau)^{{1\over m}-2} \ge Crs_0^{1\over m}
$$
for $\tau\in [c'_1,c'_2]$. Since $v_{\mbox{-}}(s)$ also satisfies
\eqref{eq:36}, Van der Corput's lemma (cf. \cite{St}) implies
\begin{eqnarray*}
|R_{0 2}
^{\mbox{-}}(t,x)|&\le& Cr^{-{n-1\over2}}s_0(rs_0^{1\over
m}) ^{-{1\over2}}\Big{(}|v_{\mbox{-}}(c'_2s_0)|
+\int_{c'_1}^{c'_2}|s_0v_{\mbox{-}}'(s_0\tau)|d\tau\Big{)}\\
&\le&Cr^{-{n-1\over2}}s_0(rs_0^{1\over m})
^{-{1\over2}}s_0^{{n+1\over2m}-1}\\
&=&Ct^{-\nu}r^{-\mu}.
\end{eqnarray*}

Since the dominated convergence theorem implies that
$J_\varepsilon(t,\cdot)$ converges (as $\varepsilon\to0$) uniformly for
$x$ in compact subsets of $\{x\in{\bf R}^n;\ |x|\ge 1\}$,
summarizing the above estimates yields
$$
|I_1(t,x)|\le Ct^{-\nu}|x|^{-\mu}\quad{\rm for}\ t\ge1\ {\rm and}\
|x|\ge 1.
$$
If $t\geq1$, $|x|\geq1$ and $t^{-1}|x|\geq 1$, then
\begin{equation}\label{I1*}
|I_1(t,x)|\le Ct^{-\nu}|x|^{-\mu}\leq
Ct^{-\frac{n}{2}}(1+t^{-1}|x|)^{-\mu}\leq Ct^{-\frac1m}(1+t^{-1}|x|)^{-\mu}.
\end{equation}
 Combining this with the estimate \eqref{eq:32} on $I_2$ (put $k=\mu-\frac1m$), we have
$$
|I(t,x)|\le Ct^{-\frac1m}(1+t^{-1}|x|)^{-\mu}\quad{\rm for}\ t\ge1, \ |x|\ge 1
 \quad{\rm and}\ t^{-1}|x|\geq 1.
$$
If $t\geq1$, $|x|\geq1$ and $t^{-1}|x|< 1$, then
$$
|I_1(t,x)|\leq Ct^{-\frac1m}\leq Ct^{-\frac1m}(1+t^{-1}|x|)^{-\mu}.
$$
Combining this with \eqref{eq:32} yields again
$$
|I(t,x)|\leq Ct^{-\frac1m}(1+t^{-1}|x|)^{-\mu}.
$$

\noindent
\underline{Case (ii):}  $t\geq 1$, $|x|\leq 1$.\\
For $I_1$ we shall prove now that
\begin{equation}\label{I1-decay}
|I_1(t,x)|\le C |t|^{-n/2} \quad{\rm for}\ |t|\ge1\ {\rm and}\ |x|\le |t|.
\end{equation}
To this end we write the integral $I_1(t,x)$ as follows:
$$
I_1(t,x)=\int_{{\bf R}^n}e^{it(P(\xi)+\langle
x/t,\xi\rangle)}\psi(P(\xi))d\xi=:\int_{{\bf R}^n}e^{it\Phi(\xi, x,t)}\psi(P(\xi))d\xi.
$$
Note that this integral and the subsequent integrations by parts can be made meaningful by inserting a series of  smooth cut-off functions $\phi(\epsilon\xi)$ for any $0<\epsilon<1$. However, this is just a technical procedure, and we refer to \cite{DY} for the details in a similar situation.

Since  $|x/t|\le 1$ and $|\nabla P(\xi)|\ge c|\xi|^{m-1}$ for large $|\xi|$,   the possible critical points satisfying
$$\nabla_\xi\Phi(\xi, x,t)=\nabla P(\xi)+x/t=0$$
must be located in some bounded ball.   In order to apply later the stationary phase principle, let $\Omega\subset {\bf R}^n$ be some open set such that ${\rm supp} \psi(P)\subset \Omega$ and $|\nabla P(\xi)|\ge c|\xi|^{m-1}$ on $\Omega$.
Note that the constant $a_1$ (from the definition of $\psi$ and Lemma \ref{lem:22}) could be increased, if necessary, such that both of those conditions can hold.
Then  we decompose $\Omega$ into $\Omega_1\cup\Omega_2$, where
 $$\Omega_1=\{\xi\in \Omega\,;\ \ |\nabla P(\xi)+\frac{x}{t}|<\frac{1}{2}|\nabla P(\xi)|+1\}$$
 and
 $$\Omega_2=\{\xi\in \Omega\,;\ \ |\nabla P(\xi)+\frac{x}{t}|>\frac{1}{4}|\nabla P(\xi)|\}.$$
Since $|\frac{x}{t}|\le 1$ and $|\nabla P(\xi)|\rightarrow\infty $ as $|\xi|\rightarrow\infty$,  $\Omega_1$ must be a bounded domain and includes all critical points of $\Phi$ inside $\Omega$.  Now we choose smooth functions $\eta_1(\xi)$ and $\eta_2(\xi)$ such that supp$\eta_j\subset \Omega_j$ and $\eta_1(\xi)+\eta_2(\xi)=1$ in $\Omega$ (e.g. see [4] for a similar  construction). And  we decompose $I_1$ as
$$I_1(t,x)=I_{11}(t,x)+I_{12}(t,x),\ \ I_{1j}(t,x):=\int_{{\bf R}^n}e^{it\Phi(\xi, x,t)}\eta_j(\xi)\psi(P(\xi))d\xi, \ j=1,2.$$

First we estimate $I_{11}$: Note that the determinant of the Hessian matrix
$$\mbox{det}(\partial_{\xi_i}\partial_{\xi_j}\Phi)_{n\times n}(\xi,x,t)
=\mbox{det}(\partial_{\xi_i}\partial_{\xi_j}P)_{n\times n}(\xi)$$
is an elliptic polynomial according to our assumption (a) and the remarks in the first paragraph of Section \ref{S2}. Hence, it is nonzero on $\Omega$ (if necessary, we can increase the value of $a_1$ to satisfy the requirement), that is,  the Hessian matrix is non-degenerate on $\Omega$.  Moreover, $|\partial_\xi^{\alpha}\Phi|\le C_\alpha$ on $\Omega_1$ for any multi-index $\alpha \in {\bf N}_0^n$.  Hence we obtain  by the stationary phase principle  that
    $$|I_{11}(t,x)|\le C|t|^{-n/2}.$$

Next we estimate $I_{12}$:  Note that $|\nabla_\xi\Phi|=|\nabla P(\xi)+\frac{x}{t}|\ge \frac{1}{4}|\nabla P(\xi)|\ge c|\xi|^{m-1}$ for $\xi\in \Omega_2$ and $|\partial_\xi^{\alpha}\Phi|\le C_\alpha |\xi|^{m-\alpha}$ for $|\alpha|\ge2$. Now we define the operator $L$ by
$$Lf:=\frac{\langle\nabla_\xi\Phi,\nabla_\xi\rangle}{it|\nabla_\xi\Phi|^2}f.
$$
Since $Le^{it\Phi}=e^{it\Phi}$, we obtain by $N$ iterated integrations by parts:
\begin{eqnarray*}
|I_{12}(t,x)|&=&\left|\int_{{\bf R}^n}e^{it\Phi(\xi, x,t)}(L^*)^N(\eta_2(\xi)\psi(P(\xi)))d\xi\right|\\
&\le& C_N|t|^{-N}\int_{{\rm supp}\psi(P)}|\xi
|^{-mN}d\xi\le C'_N|t|^{-N},
\end{eqnarray*}
where $N>n$ and $L^*$ is the adjoint operator of $L$.
Combining the two cases yields the claimed estimate $|I_1|\le C|t|^{-n/2}$ for $|t|>1$ and $|x|\le |t|$.\\

Together with the estimate \eqref{eq:32} (with $k=\mu-\frac1m$) on $I_2$ this yields
\begin{equation} \label{eq:38}
|I(t,x)|\leq Ct^{-\frac1m}(1+t^{-1}|x|)^{-\mu}\quad{\rm for}\ t\ge1\ {\rm and}\ x\in{\bf R}^n.
\end{equation}

\noindent
\underline{Case (iii):}  $t\in(0,1)$ and $x\in{\bf R}^n$.\\
Here,  we observe that
$$
\int_{{\bf R}^n}e^{i(\langle x,\xi\rangle+tP(\xi))}d\xi =t^{-{n\over
m}}\int_{{\bf R}^n} e^{i(\langle t^{-{1\over
m}}x,\xi\rangle+tP(t^{-{1\over m}}\xi))}d\xi.
$$
Let $P_t(\xi)=tP(t^{-{1\over m}}\xi)$, $\rho_t(s,\omega)=t^{1\over
m}\rho({s\over t},\omega)$, and $\sigma_t(s,\omega)=t^{1\over
m}\sigma({s\over t},\omega)$, then \eqref{eq:21} still holds with $P$,
$\rho$, $\sigma$ replaced respectively by $P_t$, $\rho_t$,
$\sigma_t$. Since it is easy to check that $\sigma_t$ also satisfies
\eqref{eq:22} with the same constants $C_{kL}$, we can deduce from \eqref{eq:38} (with
$t=1$) that
\begin{equation} \label{eq:39}
|I(t,x)|\le Ct^{-{n\over m}}(1+t^{-{1\over m}}|x|)^{-\mu} \quad{\rm
for}\ t\in(0,1)\ {\rm and}\ x\in{\bf R}^n.
\end{equation}
And the proof for negative $t$ is analogous.  This completes the proof of
the theorem.
\end{proof}

\begin{rem} \label{rem:32}
If $P$ is homogeneous and non-degenerate, then by scaling the estimates \eqref{eq:31}, one recovers the
following sharp form in the $(t, x)$-variables (see \cite{ZYF}):
$$
|{\F}^{-1}(e^{itP(\xi)})(x)|\leq Ct^{-\frac{n}{m}}(1+|t|^{-1/m}|x|)^{-\mu}
 \ {\rm for} \ t\neq 0 .
$$
In particular, we remark that the index  $\mu={n(m-2)\over 2(m-1)}$ is optimal by testing the
special case $e^{i|\xi|^m}$. In fact, from Proposition 5.1(ii) in \cite{Miy}, p. 289, there exists a
positive constant $c$ such that
$$
|{\F}^{-1}(e^{i|\cdot|^{m}})(x)|\ge c(1+|x|)^{-\mu}
 \ {\rm for} \  x\in {\bf R}^{n} .
$$

\end{rem}

\begin{rem} \label{rem:33a}
The decay estimate \eqref{eq:310a} on $I_2$ can be improved under the additional assumption that $P(\xi)$ has only non-degenerate critical points (or, equivalently, for a nonzero Gaussian curvature of the hypersurface $S$) inside the support of $(1-\psi(P(\xi))$. Then, Theorem 1 of  \S~VIII.3 in \cite{St} implies:
$$
  |I_2(t,x)|\le C(1+|t|+|x|)^{-\frac{n}{2}}\qquad \forall t,\,x.
$$
E.g., this assumption holds if $m=2$ or in the example $P(\xi)=|\xi|^4+|\xi|^2$.

An intermediate decay result for $I_2$ holds, if the Hessian of $P$ has at least rank $k$ ($1\le k\le n$)
inside the support of $(1-\psi(P(\xi))$ (or, eqivalently, if $S$ has at least $k$ nonzero principal curvatures there). Then we have  $I_2=\mathcal O\left((1+|t|+|x|)^{-k/2}\right)$ by Littman's Theorem (cf.\  \S~VIII.5.8 in \cite{St}).

\end{rem}

\begin{rem} \label{rem:34}
An analogous method as above leads to
$$
|\partial^{\alpha} I(t,x)|=|{\F}^{-1}({\xi}^{\alpha}
e^{itP(\xi)})(x)|\le
\left\{
\begin{array}{ll}
    C|t|^{-{n\over m}}(1+|t|^{-{1\over m}}|x|)^{-\mu} \ {\rm for}\ \ 0<|t|\leq1, \\
    C|t|^{-{1\over m}}(1+|t|^{-1}|x|)^{-\mu} \ {\rm for}\ \ |t|\ge 1, \    \\
\end{array}
\right.
$$
where $\alpha\in {\bf Z}_{+}^n, |\alpha|=b$, $0\leq
b\leq\frac{mn-2n}{2}$ and $\mu=\frac{mn-2n-2b}{2(m-1)}$.

\end{rem}

\section{Decay/growth estimates for Schr\"odinger equations}
\label{S4}

Here we shall apply Theorem \ref{thm:31} to
establish $L^p-L^q$ estimates for \eqref{eq:11}.
Since $P(D)$ is self-adjoint in $L^2({\bf R}^n)$, we have
$\|e^{itP(D)}u_0\|_{L^2}=\|u_0\|_{L^2}$ for all $0\le|t|<\infty$ by Stone's
theorem. Next we define the following set of admissible index pairs:
$$
\triangle\!:=\!\{(p,q);\ \mbox{$({1\over p},{1\over q})$}\ {\rm lies\ in\
the\ closed\ quadrilateral}   \;ABCD\ {\rm subtracting \ the \ apex}\
A \},
$$
where $A=({1\over 2},{1\over 2})$, $B=(1,{1\over\tau})$, $C=(1,0)$,
and $D=({1\over\tau'},0)$ for $\tau={2(m-1)\over m-2}$ and
${1\over\tau}+{1\over\tau'}=1$. Moreover, we denote by $H^1$ the
Hardy space on ${\bf R}^n$ and by BMO the space of functions with
bounded mean oscillation on ${\bf R}^n$.

\begin{thm} \label{thm:35}
 Let the assumption of Theorem \ref{thm:31} be satisfied.
Then
\begin{equation} \label{eq:311}
\|e^{itP(D)}u_0\|_{L^q}\le
\left\{
\begin{array}{ll}
    C|t|^{{n\over m}({1\over q}-{1\over p})}\|u_0\|_{L^p} &\hbox{\rm for}\ \ 0<|t|\leq1, \\
    C|t|^{n|\frac{1}{q}-{1\over p'}|-\frac1m}\|u_0\|_{L^p} &\hbox{\rm for}\ \ |t|\ge 1,\\
\end{array}
\right.
\end{equation}
where $(p,q)\in\triangle$, but $(p,q)\ne(1,\tau),\ (\tau',\infty)$.
When $(p,q)=(1,\tau)$ (resp. $(\tau',\infty)$), \eqref{eq:311} still holds
if $L^1$ (resp. $L^\infty$) is replaced by $H^1$ (resp. BMO).
\end{thm}

\begin{proof}
 When $({1\over p},{1\over q})$ lies in the edge BC,
but $({1\over p},{1\over q})\ne{\rm B}$ (i.e., $p=1$ and
$\tau<q\le\infty$), it follows from Young's inequality and Theorem
\ref{thm:31} that
\begin{equation} \label{eq:312}
\|e^{itP(D)}u_0\|_{L^q}\le\| {\F}^{-1}(e^{itP})\|_{L^q} \|u_0\|_{L^1}\le
\left\{
\begin{array}{ll}
    C|t|^{{n\over m}({1\over q}-1)}\|u_0\|_{L^1} &\hbox{\rm for}\ \ 0<|t|\leq1, \\
    C|t|^{\frac{n}{q}-\frac1m}\|u_0\|_{L^1} &\hbox{\rm for}\ \ |t|\ge 1.\\
\end{array}
\right.
\end{equation}

When $({1\over p},{1\over q})={\rm B}$ (i.e., $p=1$ and $q=\tau$),
this estimate (with $L^1$ replaced by $H^1$) follows from the
boundedness of the Riesz potential $I_{n/\tau'}$ (cf. \cite{St}, p.136).
This proves the points $(1,{1\over q})$ in the side $\overline{CB}$.
Now in view of \eqref{eq:312}, by the Marcinkiewicz
interpolation theorem (see \cite{G}, p.56), we can conclude the proof of
\eqref{eq:311} for the points in the closed triangle $ABC$. Next, by duality the
desired arguments for the triangle $ADC$ follow immediately from
the results in the triangle $ABC$. This completes the
 proof of the theorem.
\end{proof}

\begin{rem} \label{rem:36}
Let $\Omega=\{\xi\in {\bf R}^n : |\xi|> a\}$ for some sufficiently
large $a$ with supp${\F}u_0\subset \Omega$. Also let $(p,q)\in\triangle$, but
$(p,q)\ne(1,\tau),\ (\tau',\infty)$.
First we note that \eqref{I1*} and \eqref{I1-decay} combine into
$$
|I_1(t,x)|\leq  Ct^{-\frac{n}{2}}(1+|t|^{-1}|x|)^{-\mu}\leq C|t|^{-{n\over m}}(1+|t|^{-{1\over m}}|x|)^{-\mu} \ {\rm for}\ \ |t|\ge 1.
$$
Similarly to the above proof, this  estimate implies
\begin{equation} \label{eq:313}
\|e^{itP(D)}u_0\|_{L^q}=\|I_1(t,\cdot )\ast u_0\|_{L^q}\le
 C|t|^{{n\over m}({1\over q}-{1\over p})}\|u_0\|_{L^p}  \  {\rm for}  \ \ |t|>0.\\
\end{equation}
When $(p,q)=(1,\tau)$ (resp.
$(\tau',\infty)$), \eqref{eq:313} still holds if $L^1$ (resp. $L^\infty$) is
replaced by $H^1$ (resp. BMO).
\end{rem}

\vskip1cm


\begin{thebibliography}{10}
\bibitem{BE}
M.~Balabane and H.~A. Emami-Rad, \emph{$L^{p}$ estimates for
Schr\"{o}dinger evolution equations}, Trans. Amer. Math. Soc.,
 \textbf{120} (1985), 357-373.
\bibitem{C1}
S.~Cui, \emph{Point-wise estimates for a class of oscillatory
integrals and related $L^{p}-L^{q}$ estimates}, J. Fourier Anal. and
Appl., \textbf{11} (2005), 441-457.
\bibitem{C2}
\bysame, \emph{Point-wise estimates for oscillatory integrals and
related $L^{p}-L^{q}$ estimates: Multi-dimensional case}, J. Fourier
Anal. and Appl., \textbf{12} (2006), 605-627.
\bibitem{DY}
Y.~Ding and X.~Yao, \emph{$L^{p}-L^{q}$ estimates for dispersive equations and related applications,} J. Math. Anal. and Appl. \textbf{356}  (2009), 711-728.
\bibitem {G}
\ L.~Grafakos, \emph{Classical and modern Fourier analysis}, Prentice Hall,  New
 Jersey, 2003
\bibitem{KPV}
C. E. ~Kenig, G. ~Ponce and L. ~Vega, \emph{Oscillatory integrals
and regularity of dispersive
         equations}, Indiana Univ. J., \textbf{40} (1991),33-69.
\bibitem{Miy}
A. ~Miyachi, \emph{On some estimates for the wave equation in $L^{p}$
and $H^{p}$}, J. Fac. Sci. Univ. Tokyo, \textbf{27} (1980), 231-354.

\bibitem{So}
C. D. ~Sogge, \emph{Fourier Integrals in Classical Analysis,}
Cambridge Univ. Press, Cambridge, 1993.
\bibitem{St}
E. M. ~Stein, \emph{Harmonic Analysis: Real Variable Method,
Orthogonality and Oscillatory Integrals}, Princeton Univ. Press, New
Jersey, 1993.
\bibitem{YZ}
X.~Yao and Q.~Zheng, \emph{Oscillatory integrals and $L^{p}$estimates for
Schr\"odinger equations}, J. Diff. Eq. \textbf{244}  (2008), 741-752.
\bibitem{ZYF}
Q. ~Zheng, X. ~Yao and D. ~Fan, \emph{Convex hypersurface and $L^{p}$
estimates for Schr\"{o}dinger equations,} J. Func. Anal.
\textbf{208} (2004), 122-139.

\end{thebibliography}
\end{document}